\documentclass[12pt]{amsart}%
\usepackage{color}
\usepackage{graphicx}
\usepackage{amsmath,amssymb,amsfonts,amsthm,amsopn}%
\usepackage{amsmath}%
\usepackage{amsfonts, mathcomp}%
\usepackage{amssymb}
\providecommand{\U}[1]{\protect\rule{.1in}{.1in}}
\newtheorem{theorem}{Theorem}[section]
\newtheorem{lemma}[theorem]{Lemma}
\newtheorem{corollary}[theorem]{Corollary}
\newtheorem{proposition}[theorem]{Proposition}
\newtheorem{definition}[theorem]{Definition}

\newtheorem{remark}[theorem]{Remark}
\numberwithin{equation}{section}

\newcommand{\beqa}{\begin{eqnarray*}}
\newcommand{\eeqa}{\end{eqnarray*}}

\newcommand{\field}[1]{\mathbb{#1}}
\newcommand{\bR}{\field{R}}

\def\la{\lambda}

\def\rd{\bR^d}

\def\rdd{{\bR^{2d}}}

\def\lrd{L^2(\rd)}

\def\<{\left<}
\def\>{\right>}

\def\mv1{M_v^1}

\def\Ren{\mathbb{R}^d}

\def\Fur{\mathcal{F}}

\def\Sn2{S_{2}(L^{2}(\Ren))}
\def\S1{S_{1}(L^{2}(\Ren))}
\def\sig00{\sigma_{0,0}}

\def\la{\langle}
\def\ra{\rangle}

\begin{document}
\title[Generalized Born--Jordan Distributions]{Generalized Born--Jordan Distributions and Applications}
\author{Elena Cordero}
\address{Dipartimento di Matematica, Universit\`a di Torino, Dipartimento di
Matematica, via Carlo Alberto 10, 10123 Torino, Italy}
\email{elena.cordero@unito.it}
\author{Maurice de Gosson}
\address{University of Vienna, Faculty of Mathematics, Oskar-Morgenstern-Platz 1 A-1090
Wien, Austria}
\email{maurice.de.gosson@univie.ac.at}
\author{Monika D\"{o}rfler}
\address{University of Vienna, Faculty of Mathematics, Oskar-Morgenstern-Platz 1 A-1090
Wien, Austria}
\email{monika.doerfler@univie.ac.at}
\author{Fabio Nicola}
\address{Dipartimento di Scienze Matematiche, Politecnico di Torino, corso Duca degli
Abruzzi 24, 10129 Torino, Italy}
\email{fabio.nicola@polito.it}
\thanks{}
\thanks{}
\thanks{}
\thanks{}
\date{}
\dedicatory{ }
\begin{abstract}
The quadratic nature of the Wigner distribution causes the appearance of
unwanted interferences. This is the reason why engineers, mathematicians and
physicists look for related time-frequency distributions, many of them are
members of the Cohen class \cite{Cohen2}. Among them, the Born-Jordan
distribution has recently attracted the attention of many authors, since the
so-called \textquotedblleft ghost frequencies \textquotedblright are damped
quite well, and the noise is in general reduced. The very insight relies on
the kernel of such a distribution, which contains the sinus cardinalis
$\mathrm{sinc}$, which can be viewed as the Fourier transform\, of the first
B-Spline $B_{1}$. Replacing the function $B_{1}$ with the spline or order $n$,
denoted by $B_{n}$, on the Fourier side we obtain $(\mathrm{sinc})^{n}$, whose
decay at infinity increases with $n$. We introduce the Cohen's Kernel
$\Theta^{n}(z_{1},z_{2})=\mathrm{sinc}^{n}(z_{1}\cdot z_{2})$ and study the
properties of the related time-frequency \, distribution $Q^{n}$, which we
call generalized Born--Jordan distribution.

\end{abstract}
\maketitle

\section{Introduction}

The study of signals in the time-frequency plane is a subject which involves
many different people: engineers, physicists, mathematicians, both on
theoretical and applied levels, cf.,e.g.,
\cite{Cohen1,Cohen2,Galleani2002,auger}.

One of the most popular time-frequency representation of a signal $f$ is the
Wigner distribution
\begin{equation}
\label{wigner}W f(x,\omega)=\int_{\mathbb{R}^{d}} f\big(x+\frac{y}%
{2}\big) \overline{f\big(x-\frac{y}{2}\big)}e^{-2\pi i y\omega}\, dy,\qquad
x,\omega\in\mathbb{R}^{d},
\end{equation}
where the signal $f$ can be thought as a function in $L^{2}(\mathbb{R}^{d})$
or more generally as a tempered distribution ($f\in\mathcal{S}^{\prime
}(\mathbb{R}^{d})$).

It is well known that the quadratic nature of this representation causes the
appearance of interferences between several components of the signal. To
overcome this issue, the so-called Cohen class was introduced in \cite{Cohen1}
and widely studied by many authors (see \cite{bogetal} and references
therein). We address the interested reader to the textbooks
\cite{Cohen2,auger}.

A member of the Cohen class $Qf$ is obtained by convolving the Wigner
representation $Wf$ with a distribution $\theta\in\mathcal{S}^{\prime
}(\mathbb{R}^{2n})$
\begin{equation}
\label{Cohenkernel}Qf=Wf\ast\theta.
\end{equation}
A possible choice is $\theta=\mathcal{F}_{\sigma}\Theta^{1}$, with
$\mathcal{F}_{\theta}\Theta^{1}$ being the symplectic Fourier transform of the
Cohen kernel
\begin{equation}
\label{sincxp}\Theta^{1}(x,\omega)=\mathrm{sinc}(x\omega)=%
\begin{cases}
\displaystyle\frac{\sin(\pi x\omega)}{\pi x\omega} & \mbox{for}\, x\omega
\neq0\\
1 & \mbox{for}\, x\omega=0
\end{cases}
\end{equation}
($x\omega= x\cdot\omega$ denoting the scalar product in $\mathbb{R}^{d}$). In
this way we obtain the Born-Jordan distribution:
\begin{equation}
\label{bj}Q^{1} f= Wf \ast\mathcal{F}_{\sigma}(\Theta^{1}),\quad
f\in L^{2}(\mathbb{R}^{d}),
\end{equation}
see \cite{bogetal, Cohen1,Cohen2,Cohenbook,cgn0,TRANSAM, golu1,auger} and the
references therein.

We can create other interesting kernels and related distributions using the
B-spline functions $B_{n}$. Recall that the sequence of B-splines
$\{B_{n}\}_{n\in\mathbb{N}_{+}}$, is defined inductively as follows. The first
B-Spline is
\[
B_{1}(t)=\chi_{\left[ -\frac12,\frac12\right] }(t),
\]
whereas, assuming that we have defined $B_{n}$, for some $n\in\mathbb{N}_{+}$,
the spline $B_{n+1}$ is defined by
\begin{equation}
\label{bsplines}B_{n+1}(t)=(B_{n}\ast B_{1})(t)=\int_{\mathbb{R}}%
B_{n}(t-y)B_{1}(y)dy=\int_{-\frac12}^{\frac12} B_{n}(t-y)dy.
\end{equation}
The spline $B_{n}$ is a piecewise polynomial of degree at most $n-1$,
$n\in\mathbb{N}_{+}$, and satisfying $B_{n}\in\mathcal{C}^{n-2}(\mathbb{R})$,
$n\geq2$. For the main properties we refer, e.g., to \cite{Christensen2016}.

Observe that $\mathrm{sinc}(\xi)=\mathcal{F} B_{1}(\xi)$ and by induction we
infer
\begin{equation}
\label{Fouriersplines}\mathrm{sinc}^{n}(\xi)=\mathcal{F} B_{n}(\xi),\quad
n\in\mathbb{N}_{+}.
\end{equation}

\begin{definition}
For $n\in\mathbb{N}$, the nth Born-Jordan kernel is the function on
${\mathbb{R}^{2d}}$ defined by
\begin{equation}
\label{nCohenkerneln}\Theta^{n}(x,\omega)=\mathrm{sinc}^{n}(x\omega
),\quad(x,\omega)\in{\mathbb{R}^{2d}}.
\end{equation}
The Born-Jordan distribution of order $n$ (BJDn) is given by
\begin{equation}
\label{e17}Q^{n} f=Wf\ast\mathcal{F}_{\sigma}(\Theta^{n}),\quad f\in L^{2}%
(\mathbb{R}^{d}).
\end{equation}
The cross-BJDn is given by
\begin{equation}
\label{crossBJ}Q^{n} (f,g)=W(f,g)\ast\mathcal{F}_{\sigma}(\Theta^{n}),\quad
f,g\in L^{2}(\mathbb{R}^{d}).
\end{equation}
We write $Q^{n} (f,f)=Q^{n} f$, for every $f\in L^{2}(\mathbb{R}^{d})$.

For $n=0$, $\Theta^{0}\equiv1$ and $\mathcal{F}_{\sigma}(1)=\delta$, so that
$Q^{0}f=Wf$, the Wigner distribution of the signal $f$.
\end{definition}

Roughly speaking, the class of the BJDn's form a subclass of the Cohen class,
containing the Wigner distribution. This subclass will play an important role
in the applications, since its members display a great capacity of damping
interferences and such reduction increases with $n$.

In this paper we show the different facets of this phenomenon, from visual
comparisons to rigorous mathematical explanations. Motivated by this issue, a
thorough study of such distributions and related pseudodifferential calculus
is performed. In this way we show the many connections and uses of the BJDn's
$Q^{n}$, paving the way to possible other interesting and useful applications.

The main subjects treated in this paper are the following:

\begin{enumerate}

\item[(i)] \emph{Regularity and Smoothness Properties of $Q^{n}$}; 

\item[(ii)] \emph{Damping of interferences in comparison with the Wigner
distribution}; 

\item[(iii)] \emph{Visual comparison in dimension $d=1$ between $Q^{n}$ and
the Wigner Distribution}; 

\item[(iv)] \emph{Born-Jordan quantization of order $n$ and related
pseudodifferential calculus.}
\end{enumerate}

The most suitable framework to handle these aspects can be found in the scale
of modulation spaces (see \cite{F1} and also the textbook \cite{grochenig}),
recalled in Subsection \ref{2.2}. In short, we first introduce another
time-frequency representation: \emph{the short-time Fourier transform (STFT)}.
Fix a Schwartz function $g\in\mathcal{S}(\mathbb{R}^{d})\setminus\{0\}$ (the
so-called \textit{window}). We define the short-time Fourier transform of $f$
as
\begin{equation}
V_{g}f(x,\omega)=\int_{\mathbb{R}^{d}}f(y)\,{\overline{g(y-x)}}\,e^{-2\pi
iy\omega}\,dy,\quad(x,\omega)\in{\mathbb{R}^{2d}}.\label{STFTdef}%
\end{equation}
For $1\leq p,q\leq\infty$, the (unweighted) modulation space $M^{p,q}%
(\mathbb{R}^{d})$ is the subspace of tempered distributions $f$ such that
\[
\Vert f\Vert_{M^{p,q}}:=\left(  \int_{\mathbb{R}^{d}}\left(  \int%
_{\mathbb{R}^{d}}|V_{g}f(x,\omega)|^{p}\,dx\right)  ^{q/p}d\omega\right)
^{1/q}<\infty\,
\]
(with obvious modifications for $p=\infty$ or $q=\infty$). Roughly speaking, a
signal $f$ is in $M^{p,q}(\mathbb{R}^{d})$ if it decays at infinity as a
function in $\ell^{p}(\mathbb{Z}^{d})$ whereas displays a smoothness measured
in the scale $\mathcal{F}L^{q}(\mathbb{R}^{d})$.

Their images under the Fourier transform are the modulation spaces
$W(\mathcal{F} L^{p},L^{q})(\mathbb{R}^{d})$ (also known as Wiener amalgam
spaces, see subsection \ref{2.2} below). Observe that a tempered distribution
$f$ in $W(\mathcal{F} L^{p},L^{q})(\mathbb{R}^{d})$ decays at infinity as a
function in $\ell^{q}(\mathbb{Z}^{d})$ whereas locally behaves as a function
in $\mathcal{F} L^{p}(\mathbb{R}^{d})$. Of particular interest is the space
$W(\mathcal{F} L^{1},L^{\infty})(\mathbb{R}^{d})$, which is an algebra under
pointwise multiplication.

With all these instruments at hand, we can exhibit the results obtained in
this paper.

\noindent\textbf{Regularity of $Q^{n}$.} It is intuitively clear that the
Born-Jordan distribution of order $n$ of a signal, as for the classical one
$Q^{1}$, is certainly not rougher than the corresponding Wigner distribution.
We shall show in Proposition \ref{pron} that the nth-Cohen kernel belongs to
the Wiener amalgam space $W(\mathcal{F} L^{1},L^{\infty})$, for every
$n\in\mathbb{N}_{+}$. This is the key tool for proving the following result
(cf. Theorem \ref{teo2}):

\begin{theorem}
\label{teo2-zero} Let $f\in\mathcal{S}^{\prime}(\mathbb{R}^{d})$ be a signal,
with $Wf\in M^{p,q}({\mathbb{R}^{2d}})$ for some $1\leq p,q\leq\infty$. Then
$Q^{n}f\in M^{p,q}({\mathbb{R}^{2d}})$, for every $n\in\mathbb{N}_{+}$.
\end{theorem}

\noindent\textbf{Damping of interferences in comparison with the Wigner
distribution.} This topic is strictly connected with the smoothness of $Q^{n}%
$, measured using the Fourier Lebesgue wave-front set. It already paved the
way for showing the smoothness of the standard Born-Jordan distribution
$Q^{1}$ in \cite{ACHA2018}. Here it is proved that it is the right instrument
to measure the smoothness of $Q^{n}$, in comparison with the Wigner
distribution, for every $n\in\mathbb{N}_{+}$.

The notion of wave-front set of a distribution is nowadays a standard
technique in the study of singularities for solutions to partial differential
equations. The basic idea is to detect the location and orientation of the
singularities of a distribution $f$ by looking at which directions the Fourier
transform of $\varphi f$ fails to decay rapidly, where $\varphi$ is a cut-off
function supported in a neighbourhood of any given point $x_{0}$. This test is
performed in the framework of edge detection, where often the Fourier
transform is replaced by other transforms, see e.g.\ \cite{kuty} and the
references therein.

We shall use the Fourier-Lebesgue wave-front set, introduced in
\cite{ptt1,ptt2,ptt3}, and related to the Fourier-Lebesgue spaces $\mathcal{F}
L^{q}_{s}(\mathbb{R}^{d})$, $s\in\mathbb{R}$, $1\leq q\leq\infty$. Recall that
the norm in the space $\mathcal{F} L^{q}_{s}(\mathbb{R}^{d})$, $1\leq
q\leq\infty$, is given by
\begin{equation}
\label{eq4-0}\|f\|_{\mathcal{F} L^{q}_{s}(\mathbb{R}^{d})}=\|\widehat{f}%
(\omega) \langle\omega\rangle^{s}\|_{L^{q}(\mathbb{R}^{d})},
\end{equation}
with $\langle\omega\rangle=(1+|\omega|^{2})^{1/2}$. Inspired by this
definition, given a distribution $f\in\mathcal{S}^{\prime}(\mathbb{R}^{d})$
its wave-front set $WF_{\mathcal{F} L^{q}_{s}} (f)\subset\mathbb{R}^{d}%
\times(\mathbb{R}^{d}\setminus\{0\})$, is the set of points $({x}_{0},{\omega
}_{0})\in\mathbb{R}^{d}\times\mathbb{R}^{d}$, ${\omega}_{0}\not =0$, where the
following condition \textit{is not satisfied}: for some cut-off function
$\varphi$ (i.e., $\varphi$ is smooth and compactly supported on $\mathbb{R}%
^{d}$), with $\varphi({x}_{0})\not =0$, and some open conic neighbourhood
$\Gamma\subset\mathbb{R}^{d}\setminus\{0\}$ of ${\omega}_{0}$ it holds
\begin{equation}
\label{eq4}\|\mathcal{F} [\varphi f](\omega) \langle\omega\rangle^{s}%
\|_{L^{q}(\Gamma)}<\infty.
\end{equation}
Observe that $WF_{\mathcal{F} L^{2}_{s}} (f)=WF_{H^{s}}(f)$ is the standard
$H^{s}$ wave-front set (see \cite[Chapter XIII]{hormander2} and Section 2
below).  Roughly speaking, $({x}_{0},{\omega}_{0})\not \in WF_{\mathcal{F}
L^{q}_{s}}(f)$ means that $f$ has regularity $\mathcal{F} L^{q}_{s}$ at
${x}_{0}$ and in the direction ${\omega}_{0}$. We are interested in the
$\mathcal{F} L^{q}_{s}$ wave-front set of the Born-Jordan distribution of
order $n$ of a given signal $f\in L^{2}(\mathbb{R}^{d})$.

Here is the mathematical explanation of the $Q^{n}$'s smoothing effects: 

\begin{theorem}
\label{mainteo}  Let $f\in\mathcal{S}^{\prime}(\mathbb{R}^{d})$ be a signal,
with $Wf\in M^{\infty,q}({\mathbb{R}^{2d}})$ for some $1\leq q\leq\infty$. Let
$({z},{\zeta})\in{\mathbb{R}^{2d}}\times{\mathbb{R}^{2d}}$, with ${\zeta
}=({\zeta}_{1},{\zeta}_{2})$ satisfying ${\zeta}_{1}\cdot{\zeta}_{2}\not =0$.
Then
\[
({z},{\zeta})\not \in WF_{\mathcal{F} L^{q}_{2n}}(Q^{n}f).
\]

\end{theorem}

This means that, if the Wigner distribution $Wf$ has local regularity
$\mathcal{F} L^{q}$ and some control at infinity, then $Q^{n}f$ is smoother,
possessing \textbf{$s=2n$ additional derivatives}, at least in the directions
${\zeta}=({\zeta}_{1},{\zeta}_{2})$ satisfying ${\zeta}_{1}\cdot{\zeta}%
_{2}\not =0$. In dimension $d=1$ this condition reduces to ${\zeta}_{1}%
\not =0$ and ${\zeta}_{2}\not =0$. Hence this result explains the smoothing
phenomenon of such distributions, which involves all the directions except
those of the coordinates axes. That is why the interferences of two components
which do not share the same time or frequency localization come out
substantially reduced. Observe that for $n=1$ we recapture the damping
phenomenon of the classical Born-Jordan distribution (cf. \cite[Theorem
1.2]{ACHA2018}).

For signal in $L^{2}(\mathbb{R}^{d})$, the previous result can be rephrased in
terms of the H\"ormander's wave-front set as follows:

\begin{corollary}
\label{cor} Let $f\in L^{2}(\mathbb{R}^{d})$, so that $Wf\in L^{2}%
({\mathbb{R}^{2d}})$. Let $(z,\zeta)$ be as in the statement of Theorem
\ref{mainteo}. Then $({z},{\zeta})\not \in WF_{H^{2n}}(Q^{n}f)$, i.e. $Q^{n}f$
has regularity $H^{2n}$ at $z$ and in the direction $\zeta$.
\end{corollary}

It seems that the smoothing effects could not occur in the directions ${\zeta
}_{1}\cdot{\zeta}_{2}=0$, as the evidences in the pictures show. From
the mathematical point of view, we provide the explanation below.

\begin{theorem}
\label{teo3} Suppose that for some $1\leq p,q_{1},q_{2}\leq\infty$,
$n\in\mathbb{N}_{+}$ and $C>0$, it occurs
\begin{equation}
\label{test}\|Q^{n}f\|_{M^{p,q_{1}}}\leq C\|W f\|_{M^{p,q_{2}}},
\end{equation}
for every $f\in\mathcal{S}(\mathbb{R}^{d})$. Then $q_{1}\geq q_{2}$.
\end{theorem}

In other terms, for a general signal, the BJDn is not everywhere smoother than
the Wigner distribution. As expected, the problems arise in the directions
$\zeta=(\zeta_{1},\zeta_{2})$ such that $\zeta_{1}\cdot\zeta_{2}=0$.

\vspace{0.5truecm} \noindent\textbf{Visual Comparison in dimension $d=1$
between $Q^{n}$ and the Wigner Distribution.}
 We now illustrate the effect of using higher order cross-term suppression by means of the generalized 
 BJDn. We display the time-frequency distributions of  both synthetic and real signals.  More precisely, Figure~\ref{fig:GenBJ} shows a comparison of the Wigner transform, the Born-Jordan transform and generalised Born-Jordan transform of the sum of  four rotated Gaussian windows. It is clearly visible, that the amount of cross-term suppression increases by applying higher-order smoothing.\\
  \begin{figure}[htbp]
\begin{center}
\hspace{-0.5cm}\includegraphics[width=1.1\textwidth]{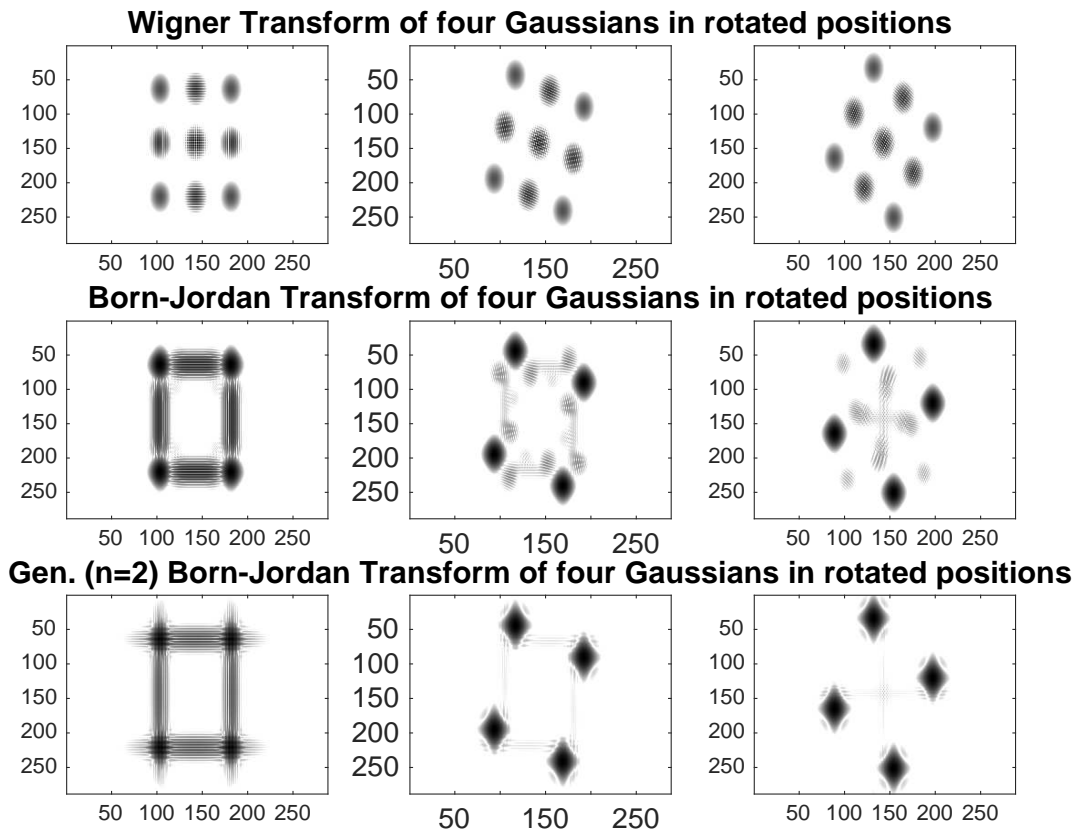}
\caption{Four Gaussian Windows in rotated positions:  Comparison of Wigner distribution, Born-Jordan and generalised Born-Jordan distribution}
\label{fig:GenBJ}
\end{center}
\end{figure}
The second example,  shown in Figure~\ref{fig:GenBJchirp}, depicts  the Wigner transform, the Born-Jordan transform and
 two versions of generalised Born-Jordan transform ($n = 10$ and $n = 100$) of another synthetic signal consisting of two linear chirps. It is notable, that the geometry of this example is different from the previous one in the sense of lacking symmetry around zero. \\
  \begin{figure}[htbp]
\begin{center}
\hspace{-0.5cm}\includegraphics[width=1.1\textwidth]{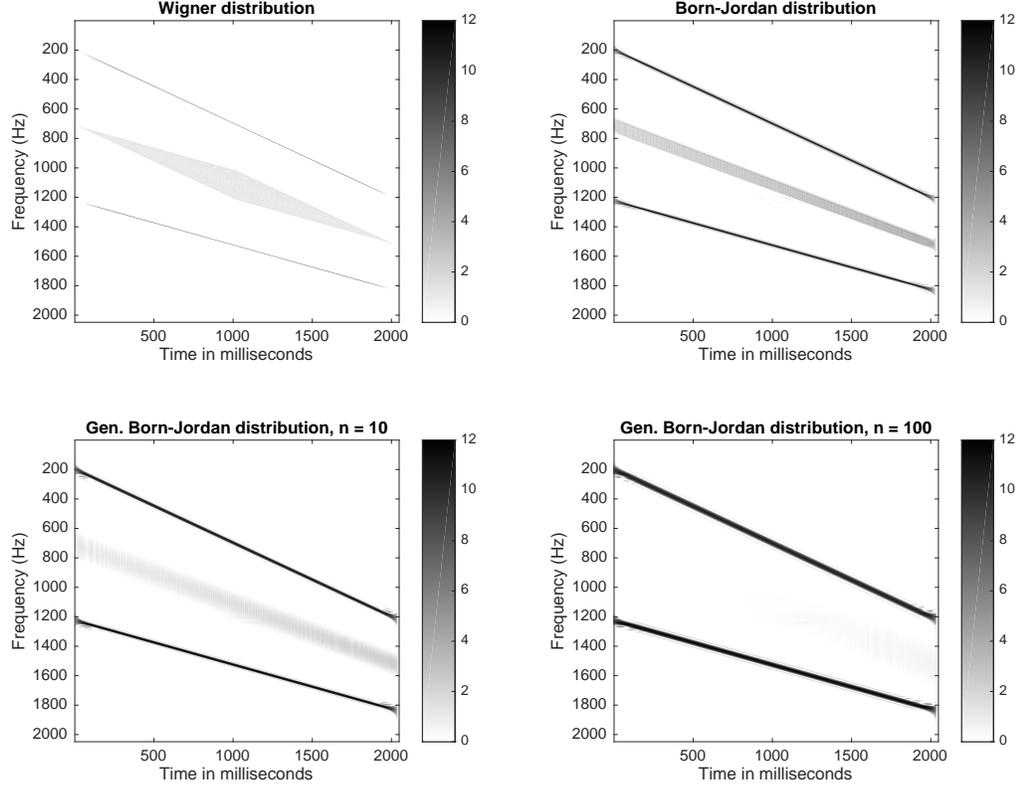}
\caption{Two linear chirps:  Comparison of Wigner distribution, Born-Jordan and generalised Born-Jordan distribution}
\label{fig:GenBJchirp}
\end{center}
\end{figure}
 As a final example, shown in Figure~\ref{fig:GenBJbat}, we applied the Wigner transform, the Born-Jordan transform and
 two versions of generalised Born-Jordan transform to a classical real signal, namely a bat call. As in the first example, the suppression of artefacts increases for exponent $n=2$, while, when applying even higher order smoothing, we observe a loss of concentration in time-frequency. As in the case of the two chirps,  the geometry of this example lacks symmetry around zero. 
 
 \begin{figure}[htbp]
\begin{center}
\hspace{-0.5cm}\includegraphics[width=1.1\textwidth]{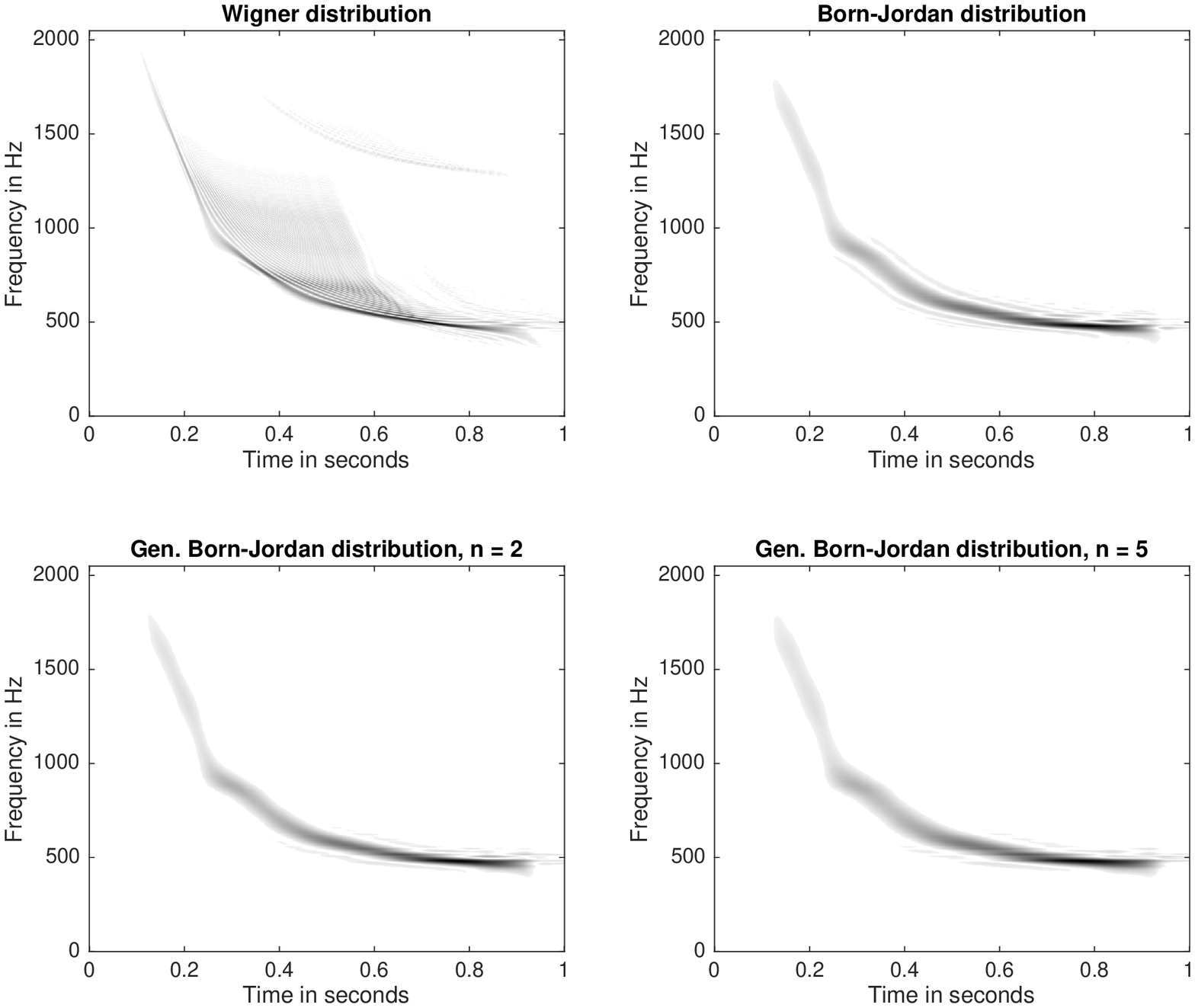}
\caption{Bat call signal: Comparison of Wigner distribution, Born-Jordan and generalised Born-Jordan distribution}
\label{fig:GenBJbat}
\end{center}
\end{figure}

 \noindent\textbf{The Born-Jordan quantization of order $n$.}
This procedure arises as the natural extension of the $n=1$ case (that is, the
classical Born-Jordan quantization). Observe that choosing $n=0$, we obtain
the Weyl quantization. We denote by $\hbar$ the reduced Planck's constant.

\begin{definition}
For $n\in\mathbb{N}$, the Born-Jordan quantization of order $n$ is the
mapping
\begin{equation}
\label{BJNQ}{a\in\mathcal{S}^{\prime}({\mathbb{R}^{2d}})}\mapsto
\widehat{A}_{\mathrm{BJ},n} =\operatorname*{Op}\nolimits_{{BJ},n}(a)=\left(
\tfrac{1}{2\pi\hbar}\right)  ^{d} \int_{\mathbb{R}^{2d}}(\mathcal{F}_{\sigma
}a)(z)\Theta^{n}(z)\widehat{T}(z) dz,
\end{equation}
where $\widehat{T}(z)=e^{-2\pi i\sigma(\widehat{z},z)}$  is the Heisenberg
operator and $\sigma$ the standard symplectic form (see the notation below).

The case $n=0$ ($\Theta^{0}\equiv1$) is the well-known Weyl quantization.
\end{definition}

Though, working in the framework of time-frequency analysis, we shall set
$\hbar=1/(2\pi)$ so that the constant in front of the integrals in
\eqref{BJNQ} disappears. 

\medskip The paper is organized as follows.

\section{Preliminaries}

\subsection{Notation}

We use $x\omega=x\cdot\omega=x_{1}\omega_{1}+\ldots+x_{d}\omega_{d}$ for the
scalar product in $\mathbb{R}^{d}$, $\langle\cdot,\cdot\rangle$ for the inner
product in $L^{2}(\mathbb{R}^{d})$ or for the duality pairing between Schwartz
functions and temperate distributions (antilinear in the second argument).
Given functions $f,g$, we write $f\lesssim g$ if $f(x)\leq C g(x)$ for every
$x$ and some constant $C>0$, and similarly for $\gtrsim$. The notation
$f\asymp g$ means $f\lesssim g$ and $f\gtrsim g$.

We write $\mathcal{C}^{\infty}_{c}(\mathbb{R}^{d})$ for the class of smooth
functions on $\mathbb{R}^{d}$ with compact support.

We denote by $\sigma$ the standard symplectic form on the phase space
$\mathbb{R}^{2d}\equiv\mathbb{R}^{d}\times\mathbb{R}^{d}$; the phase space
variable is denoted $z=(x,\omega)$ and the dual variable by $\zeta=(\zeta
_{1},\zeta_{2})$. By definition $\sigma(z,\zeta)=Jz\cdot\zeta=\omega\cdot
\zeta_{1}-x\cdot\zeta_{2}$, where
\[
J=%
\begin{pmatrix}
0_{d\times d} & I_{d\times d}\\
-I_{d\times d} & 0_{d\times d}%
\end{pmatrix}
.
\]

The Fourier transform of a function $f(x)$ in $\mathbb{R}^{d}$ is
\[
\mathcal{F} f(\omega)=\widehat{f}(\omega)= \int_{\mathbb{R}^{d}} e^{-2\pi i
x\omega} f(x)\, dx,
\]
and the symplectic Fourier transform of a function $F(z)$ in the phase space
$\mathbb{R}^{2d}$ is defined by
\[
\mathcal{F}_{\sigma}F(\zeta)=\int_{{\mathbb{R}^{2d}}} e^{-2\pi{i}\sigma
(\zeta,z)} F(z)\, dz.
\]
The symplectic Fourier transform is an involution, i.e., $\mathcal{F}_{\sigma
}(\mathcal{F}_{\sigma}F)=F$. Moreover, $\mathcal{F}_{\sigma}F(\zeta)=
\mathcal{F} F(J \zeta)$.

Observe that $\Theta^{n}(J(\zeta_{1},\zeta_{2}))=\Theta^{n}(\zeta_{1}%
,\zeta_{2})$ so that
\begin{equation}
\label{cft}\mathcal{F}_{\sigma}(\Theta^{n})=\mathcal{F} (\Theta^{n}%
),\quad\forall n\in\mathbb{N}_{+}.
\end{equation}

For $s\in\mathbb{R}$ the $L^{2}$-based Sobolev space $H^{s}(\mathbb{R}^{d})$
is constituted by the distributions $f\in\mathcal{S}^{\prime}(\mathbb{R}^{d})$
such that
\begin{equation}
\label{normhs}\|f\|_{H^{s}}:=\| \widehat{f}(\omega) \langle\omega\rangle^{s}
\|_{L^{2}}<\infty.
\end{equation}

\subsection{Time-frequency representations and main properties}

\subsubsection{Wigner distribution and ambiguity function
\cite{Birkbis,grochenig}}

We already defined in Introduction, see \eqref{wigner}, the Wigner
distribution $Wf$ of a signal $f\in\mathcal{S}^{\prime}(\mathbb{R}^{d})$. In
general, we have $Wf\in\mathcal{S}^{\prime}({\mathbb{R}^{2d}})$. When $f\in
L^{2}(\mathbb{R}^{d})$ we have $Wf\in L^{2}({\mathbb{R}^{2d}})$ and in fact it
turns out
\begin{equation}
\label{wigner2}\|Wf\|_{L^{2}({\mathbb{R}^{2d}})}=\|f\|_{L^{2}(\mathbb{R}^{d}%
)}^{2}.
\end{equation}
In the sequel we will encounter several times the symplectic Fourier transform
of $W f$, which is known as the (radar) \textit{ambiguity function} $Af$. We
have the formula
\begin{equation}
\label{ambiguity}Af(\zeta_{1},\zeta_{2})=\mathcal{F}_{\sigma}Wf (\zeta
_{1},\zeta_{2})= \int_{\mathbb{R}^{d}} f\big(y+\frac{1}{2}\zeta_{1}%
\big) \overline{f\big(y-\frac{1}{2}\zeta_{1}\big)}e^{-2\pi i\zeta_{2} y}\, dy.
\end{equation}
We refer to \cite[Chapter 9]{Birkbis} and in particular to \cite[Proposition
175]{Birkbis} for more details.

\subsubsection{Marginal properties of $Q^{n}$}

The members of the Cohen class are also called pseudo-density functions since
they are supposed to indicate how the signal density is distributed over time
and frequency. The terminology \emph{pseudo-density} is due to the fact that
such distributions in general are not positive functions and can take not only
negative but even complex values. In order $Q^{n}$ to be a pseudo-density
function, it must satisfy certain requirements. In particular, the marginal
densities
\begin{equation}
\label{pd}\int_{\mathbb{R}^{d}} Q^{n} f(x,\omega) d\omega= |f(t)|^{2}%
,\quad\int_{\mathbb{R}^{d}} Q^{n} f(x,\omega) dx=|\hat{f}(\omega)|^{2},
\end{equation}
for every $f$ in the Schwartz class $\mathcal{S}(\mathbb{R}^{d})$. It can be
shown (see \cite{Janssen82} or \cite[Proposition 97]{deGossonWigner}) that
those conditions are equivalent to the requirements
\begin{equation}
\label{mdkernel}\mathcal{F} (\Theta^{n})(x, 0)=1,\,\,\forall x\in
\mathbb{R}^{d},\quad\mathcal{F} (\Theta^{n})(0, \omega)=1,\,\forall\omega
\in\mathbb{R}^{d}.
\end{equation}
In this case, using \eqref{cft}, \eqref{e17} and \eqref{nCohenkerneln}, one
sees that are trivially satisfied, since sinc$^{n}(0)=1$, for every
$n\in\mathbb{N}$.

\subsubsection{The Moyal identity is not satisfied}

In quantum mechanics (but perhaps not really necessary for signal analysis, as
already stated by Janssen in \cite{Janssen82}) a quite convenient property for
Cohen's kernels \eqref{Cohenkernel} is the validity of Moyal's formula
(\cite[Theorem 14.2 and 27.15]{deBruijn}):
\begin{equation}\label{Moyal}
\la Q(f_1,g_1), Q(f_2,g_2)\ra _{L^2(\rdd)}=\la f_1,f_2\ra_{L^2(\rd)}\overline{\la g_1,g_2\ra}_{L^2(\rd)},\quad f_1,f_2,g_1,g_2\in\lrd.
\end{equation}
The Wigner distribution, as well as the STFT and the ambiguity function
satisfy \eqref{Moyal}. \emph{Though, the BJDn $Q^{n}$, for $n\in\mathbb{N}_{+}$, does not fulfil Moyal's identity}. To prove this statement, we use the
following characterization (cf. \cite[Section 3]{Janssen82} and \cite{gobook16}):
\begin{proposition}
A member of the Cohen class, cf. \eqref{Cohenkernel}, satisfies Moyal's
formula \eqref{Moyal} if and only if
\begin{equation}
\label{Mker}|\theta(x,\omega) |=1, \quad\mbox{for\, all}\quad(x,\omega
)\in{\mathbb{R}^{2d}}.
\end{equation}

\end{proposition}

Choosing $Q=Q^{n}$, $n\in\mathbb{N}_{+}$, we have $\theta(x,\omega
)=\mathrm{sinc}^{n}(x\omega)$, so that condition \eqref{Mker} is not satisfied
for any $n\in\mathbb{N}_{+}$. Observe that for $n=0$ (the Wigner distribution)
the previous conditions holds, as expected.

\subsection{Modulation spaces
\cite{Birkbis,feichtinger80,feichtinger83,feichtinger90,grochenig}}

\label{2.2} Modulation spaces are used in Time-frequency Analysis to measure
the time-frequency concentration of a signal. As already observed in the
introduction, their construction relies on the notion of short-time (or
windowed) Fourier transform defined in \eqref{STFTdef}.

Let now $s\in\mathbb{R}$, $1\leq p,q\leq\infty$. The \textit{ modulation
space} $M^{p,q}_{s}(\mathbb{R}^{d})$ consists of all tempered distributions
$f\in\mathcal{S}^{\prime}(\mathbb{R}^{d}) $ such that
\begin{equation}
\label{defmod}\|f\|_{M^{p,q}_{s}}:=\left( \int_{\mathbb{R}^{d}} \left(
\int_{\mathbb{R}^{d}}|V_{g}f(x,\omega) |^{p}\langle\omega\rangle^{sp}\,
dx\right) ^{q/p}d\omega\right) ^{1/q}<\infty\,
\end{equation}
(with obvious changes for $p=\infty$ or $q=\infty$).  When $s=0$ we write
$M^{p,q}(\mathbb{R}^{d})$ instead of $M^{p,q}_{0}(\mathbb{R}^{d})$. The
shorthand notation for $M^{p,p}_{s}(\mathbb{R}^{d})$ is $M^{p}_{s}%
(\mathbb{R}^{d})$. The spaces $M^{p,q}_{s}(\mathbb{R}^{d})$ are Banach spaces
for any $1\leq p,q\leq\infty$, and  every non-zero $g\in\mathcal{S}%
(\mathbb{R}^{d})$ yields an equivalent norm in  \eqref{defmod}.

Modulation spaces generalize and include as special cases several function
spaces arising in Harmonic Analysis. In particular for $p=q=2$ we have
\[
M^{2}_{s}(\mathbb{R}^{d})=H^{s}(\mathbb{R}^{d}),
\]
whereas $M^{1}(\mathbb{R}^{d})$ coincides with the Segal algebra
$S_{0}(\mathbb{R}^{d})$ (cf. \cite{fei0}), and $M^{\infty,1}(\mathbb{R}^{d})$
is the so-called Sj\"ostrand class \cite{charly06}.

As already observed in the Introduction, in the notation $M^{p,q}_{s}$ the
exponent $p$ is a measure of decay at infinity (on average) in the scale of
spaces $\ell^{p}$, whereas the exponent $q$ is a measure of smoothness in the
scale $\mathcal{F} L^{q}$. The number $s$ is a further regularity index,
completely analogous to that appearing in the Sobolev spaces $H^{s}%
(\mathbb{R}^{d})$.

Other modulation spaces, also known as \textit{Wiener amalgam spaces}, are
obtained by exchanging the order of integration in \eqref{defmod}. Precisely,
the modulation spaces $W(\mathcal{F}L^{p},L^{q})(\mathbb{R}^{d})$, for
$p,q\in\lbrack1,+\infty)$, is given by the distributions $f\in\mathcal{S}%
^{\prime}(\mathbb{R}^{d})$ such that
\[
\Vert f\Vert_{W(\mathcal{F}L^{p},L^{q})(\mathbb{R}^{d})}:=\left(
\int_{\mathbb{R}^{d}}\left(  \int_{\mathbb{R}^{d}}|V_{g}f(x,\omega
)|^{p}\,d\omega\right)  ^{q/p}dx\right)  ^{1/q}<\infty\,
\]
(with obvious changes for $p=\infty$ or $q=\infty$). Using Parseval identity
in \eqref{STFTdef}, we can write the so-called fundamental identity of
time-frequency analysis\thinspace\
\[
V_{g}f(x,\omega)=e^{-2\pi ix\omega}V_{\hat{g}}\hat{f}(\omega,-x),
\]
hence
\[
|V_{g}f(x,\omega)|=|V_{\hat{g}}\hat{f}(\omega,-x)|=|\mathcal{F}(\hat
{f}\,T_{\omega}\overline{\hat{g}})(-x)|
\]
so that
\[
\Vert f\Vert_{{M}^{p,q}}=\left(  \int_{\mathbb{R}^{d}}\Vert\hat{f}\ T_{\omega
}\overline{\hat{g}}\Vert_{\mathcal{F}L^{p}}^{q}\ d\omega\right)  ^{1/q}%
=\Vert\hat{f}\Vert_{W(\mathcal{F}L^{p},L^{q})}.
\]
This means that such Wiener amalgam spaces can be viewed as the image under
Fourier transform\thinspace\ of modulation spaces: $\mathcal{F}({M}%
^{p,q})=W(\mathcal{F}L^{p},L^{q})$.

We will frequently use the following product property of Wiener amalgam spaces
(\cite[Theorem 1 (v)]{feichtinger80}): For $1\leq p,q\leq\infty$,
\begin{equation}\label{product}
\textit{if $f\in W(\Fur L^1,L^\infty)$ and $g\in W(\Fur L^p,L^q)$ then $fg\in W(\Fur L^p,L^q)$}.
\end{equation}
Taking $p= 1, q=\infty$, we obtain that $W(\mathcal{F} L^{1},L^{\infty
})({\mathbb{R}^{2d}})$ is an algebra under point-wise multiplication.

\begin{proposition}
\label{c1} Let $1\leq p,q\leq\infty$ and $A\in GL(d,\mathbb{R})$. Then, for
every $f\in W(\mathcal{F}L^{p},L^{q})(\mathbb{R}^{d})$,
\begin{equation}
\label{dilAW0}\|f(A\,\cdot)\|_{W(\mathcal{F}L^{p},L^{q})}\leq C |\det
A|^{(1/p-1/q-1)}(\det(I+A^{*} A))^{1/2}\|f\|_{W(\mathcal{F}L^{p},L^{q})}.
\end{equation}
In particular, for $A=\lambda I$, $\lambda>0$,
\begin{equation}
\label{dillambda}\|f(A\,\cdot)\|_{W(\mathcal{F}L^{p},L^{q})}\leq C
\lambda^{d\left( \frac1p-\frac1q-1\right) }(\lambda^{2}+1)^{d/2}
\|f\|_{W(\mathcal{F}L^{p},L^{q})}.
\end{equation}

\end{proposition}

In the proof of Theorem \ref{teo3} we will use dilation properties of
Gaussians (first proved in \cite[Lemma 1.8]{toft}, see also \cite[Lemma
3.2]{CNJFA2008}):

\begin{lemma}
\label{lemma5.2-zero} Let $\varphi(x)=e^{-\pi|x|^{2}}$ and $\lambda>0$. Then
\[
\|\varphi(\lambda\,\cdot)\|_{M^{p,q}}\asymp\lambda^{-d/q^{\prime}}%
\quad\mathrm{as}\ \lambda\to+\infty.
\]

\end{lemma}

\subsection{Wave-front set for Fourier-Lebesgue spaces \cite{hormander2,ptt1}}

The notion of $H^{s}$ wave-front set allows to quantify the regularity of a
function or distribution in the Sobolev scale, at any given point and
direction. This is done by microlocalizing the definition of the $H^{s}$ norm
in \eqref{normhs} as follows (cf.\ \cite[Chapter XIII]{hormander2}).

Given a distribution $f\in\mathcal{S}^{\prime}(\mathbb{R}^{d})$ we define its
wave-front set $WF_{H^{s}} (f)\subset\mathbb{R}^{d}\times(\mathbb{R}%
^{d}\setminus\{0\})$, as the set of points $({x}_{0},{\omega}_{0}%
)\in\mathbb{R}^{d}\times\mathbb{R}^{d}$, ${\omega}_{0}\not =0$, where the
following condition is \textit{not} satisfied: for some cut-off function
$\varphi\in C^{\infty}_{c}(\mathbb{R}^{d})$ with $\varphi({x}_{0})\not =0$ and
some open conic neighborhood of $\Gamma\subset\mathbb{R}^{d}\setminus\{0\}$ of
${\omega}_{0}$ we have
\[
\|\mathcal{F} [\varphi f](\omega) \langle\omega\rangle^{s}\|_{L^{2}(\Gamma
)}<\infty.
\]
More in general one can start from the Fourier-Lebesgue spaces $\mathcal{F}
L^{q}_{s}(\mathbb{R}^{d})$, $s\in\mathbb{R}$, $1\leq q\leq\infty$, which is
the space of distributions $f\in\mathcal{S}^{\prime}(\mathbb{R}^{d})$ such
that the norm in \eqref{eq4-0} is finite. Arguing exactly as above (with the
space $L^{2}$ replaced by $L^{q}$) one then arrives in a natural way to a
corresponding notion of wave-front set $WF_{\mathcal{F} L^{q}_{s}}(f)$ as we
anticipated in Introduction (see \eqref{eq4}).

For our purpose we need to recall some basic results about the action of
constant coefficient linear partial differential operators on such wave-front
set (cf. \cite{ptt1}).

Given the operator
\[
P=\sum_{|\alpha|\leq m} c_{\alpha}\partial^{\alpha},\quad c_{\alpha}%
\in\mathbb{C};
\]
it is straightforward to see that, for $1\leq q\leq\infty$, $s\in\mathbb{R}$,
$f\in\mathcal{S}^{\prime}(\mathbb{R}^{d})$,
\[
WF_{\mathcal{F} L^{q}_{s}}(Pf) \subset WF_{\mathcal{F} L^{q}_{s+m}}(f).
\]
Consider now the inverse inclusion. We say that $\zeta\in\mathbb{R}^{d}$,
$\zeta\not =0$, is non characteristic for the operator $P$ if
\[
\sum_{|\alpha|=m} c_{\alpha}\zeta^{\alpha}\not =0.
\]
This means that $P$ is elliptic in the direction $\zeta$. The following result
is a microlocal version of the classical regularity result of elliptic
operators (see \cite[Corollary 1 (2)]{ptt1}):

\begin{proposition}
\label{pro3} Let $1\leq q\leq\infty$, $s\in\mathbb{R}$ and $f\in
\mathcal{S}^{\prime}(\mathbb{R}^{d})$. Let $z\in\mathbb{R}^{d}$ and suppose
that $\zeta\in\mathbb{R}^{d}\setminus\{0\}$ is non characteristic for $P$.
Then, if $(z,\zeta)\not \in WF_{\mathcal{F} L^{q}_{s}}(Pf)$ we have
$(z,\zeta)\not \in WF_{\mathcal{F} L^{q}_{s+m}}(f)$.
\end{proposition}

\section{Generalized Born--Jordan Kernels for Monomials}

\label{3}  Let $\mathbb{C}[x,\omega]$ be the commutative ring of polynomials
generated by  $x$ and $\omega$; it consists of all finite sums $a(x,\omega
)=\sum\lambda_{m\ell}a_{m\ell}(x,\omega)$ ($\lambda_{m\ell}\in\mathbb{C}$)
where  $a_{m\ell}(x,\omega)=\omega^{m}x^{\ell}$ with $(m,\ell)\in
\mathbb{N}^{2}$. We  identify $\mathbb{C}[x,\omega]$ with the ring of
polynomial functions in the  variables $(x,\omega)\in\mathbb{R}^{2}$. We
denote by $\mathbb{C} [\widehat{x},\widehat{\omega}]$ the corresponding Weyl
algebra; it is realized  as the non-commutative unital algebra generated by
the two operators  $\widehat{x}$ and $\widehat{\omega}$ satisfying
$[\widehat{x},\widehat{\omega}]=(i/2\pi)I_{\mathrm{d}}$. These operators are
realized as the unbounded  operators defined on $L^{2}(\mathbb{R)}$ by
$\widehat{x}f=xf$ and  $\widehat{\omega}f=-(i/2\pi)\partial_{x}f$. We will
call \emph{quantization of} $\mathbb{C}[x,\omega]$ any continuous linear
mapping $\operatorname*{Op}:$  $\mathbb{C}[x,\omega]\longrightarrow
\mathbb{C}[\widehat{x},\widehat{\omega}]$  having the following properties:

\begin{enumerate}

\item[(Q1)] Triviality: $\operatorname*{Op}(1)=I_{\mathrm{d}}$,
$\operatorname*{Op}(x)=\widehat{x}$, and $\operatorname*{Op}(\omega
)=\widehat{\omega}$;

\item[(Q2)] Dirac's restricted rule:
\[
\lbrack x,\operatorname*{Op}(a_{m\ell})]=(i/2\pi)\operatorname*{Op}%
(\{x,a_{m\ell}\})\text{ \ , \ }[\omega,\operatorname*{Op}(a_{m\ell}%
)]=(i/2\pi)\operatorname*{Op}(\{\omega,a_{m\ell}\});
\]

\item[(Q3)] Self-adjointness: If $a\in\mathbb{C}[x,\omega]$ then
$\operatorname*{Op}(a)$ is self-adjoint on its domain. 
\end{enumerate}

One shows \cite{dogo15} (also see \cite{Cohenbook}) that for every
quantization of  $\mathbb{C}[x,\omega]$ there exists \cite{Cohenbook,dogo15} a
function $f$ with  $f(0)=1$ and $e^{-it/2}f$ real such that
\begin{equation}
\operatorname*{Op}(a_{m\ell})=\sum_{j=0}^{\min(m,\ell)}j!\binom{m}{j}
\binom{\ell}{j}f^{(j)}(0)(2\pi)^{-j}\widehat{\omega}^{m-j}\widehat{x}^{\ell
-j}. \label{opaml}%
\end{equation}
Let $(a_{m\ell})_{\sigma}=\mathcal{F}_{\sigma}a_{m\ell}$ be the symplectic
Fourier  transform of $a_{m\ell}$ and $\widehat{T}(z)=e^{-2\pi i\sigma
(\widehat{z},z)}$  the Heisenberg operator.

\begin{proposition}
Let $\operatorname*{Op}:\mathcal{S}^{\prime}(\mathbb{R}^{2n})\longrightarrow
\mathcal{L}(\mathcal{S}(\mathbb{R}^{n}),\mathcal{S}^{\prime}(\mathbb{R}%
^{n}))$  be a quantization having the properties (Q1), (Q2), (Q3). (i) The
restriction  of $\operatorname*{Op}$ to $\mathbb{C}[x,\omega]$ is then given
by
\begin{equation}
\operatorname*{Op}(a_{m\ell})=\int(a_{m\ell})_{\sigma}(x,\omega)\Phi
(2\pi\omega x)\widehat{T}(x,\omega)d\omega dx \label{cohen1}%
\end{equation}
where $\Phi(t)=e^{-it/2}f(t)$. (ii) The Cohen kernel $\theta$ of
$\operatorname*{Op}$ thus has symplectic Fourier transform $\mathcal{F}%
_{\sigma}\theta$  given by
\begin{equation}
\mathcal{F}_{\sigma}\theta(x,\omega)=\Phi(2\pi\omega x). \label{thesig}%
\end{equation}

\end{proposition}

\begin{proof}
A detailed proof is given in by Domingo and Galapon \cite{dogo15} (formulas
(10) and (14)). Notice that formula (\ref{cohen1}) readily follows from
(\ref{opaml}) using the elementary formula
\[
\mathcal{F}(\omega^{m}\otimes x^{\ell})=(i/2\pi)^{m+\ell}\delta^{(m)}%
(\omega)\otimes\delta^{(\ell)}(x).
\]
Formula (\ref{thesig}) follows since (\ref{cohen1}) is the Weyl representation
of the operator with twisted symbol $(a_{m\ell})_{\sigma}\Phi$ [the twisted
symbol is the symplectic Fourier transform of the usual symbol].
\end{proof}

\begin{remark}
This result shows that if one limits oneself to pseudo-differential calculi
satisfying the Dirac conditions (Q2) then the Cohen kernel is of a very
particular type: its Fourier transform only depends on the product $\omega
x$.  In particular, the associated quasidistribution $Q\psi=W\psi\ast\theta$
satisfies the marginal conditions since $\mathcal{F}_{\sigma}\theta
(0)=\Phi(0)=1$ (see  \cite{gobook16}, formula (7.29), p. 107). 
\end{remark}

We now focus on the case the symplectic Fourier transform of the Cohen kernel
is given by
\[
\mathcal{F}_{\sigma}\theta(x,\omega)=\operatorname{sinc}^{n}(\pi\omega
x)\text{ , } n\in\mathbb{N}=\{0,1,2,...\}.
\]
With the notation above we thus have $\Phi(\pi\omega x)=\operatorname{sinc}
^{n}(\pi\omega x)$ so that $\Phi(t)=\operatorname{sinc}^{n}(t/2)$ and hence
$f(t)=e^{it/2}\operatorname{sinc}^{n}(t/2)$. Suppose first $n=0$; then
$f^{(j)}(0)=(i/2)^{j}$ hence formula (\ref{opaml}) yields
\[
\operatorname*{Op}(a_{m\ell})=\sum_{j=0}^{\min(m,\ell)}\binom{m}{j}\binom
{\ell}{j}j!\left(  \frac{i}{4\pi}\right)  ^{j}\widehat{\omega}^{m-j}
\widehat{x}^{\ell-j}
\]
so that $\operatorname*{Op}(a_{m\ell})=\operatorname*{Op}^{\mathrm{W}
}(a_{m\ell})=\operatorname*{Op}_{{BJ,0}}(a_{m\ell})$ (see \eqref{BJNQ}) is
just the Weyl ordering of the monomial $a_{m\ell}$  (\cite{dogo15} and
\cite{gobook16}, p.34). Suppose next $n=1$. Then  $f^{(j)}(0)=i^{j}/(j+1)$
and
\[
\operatorname*{Op}(a_{m\ell})=\sum_{j=0}^{\min(m,\ell)}\binom{m}{j}\binom
{\ell}{j}\frac{j!}{j+1}\left(  \frac{i}{2\pi}\right)  ^{j}\widehat{\omega
}^{m-j}\widehat{x}^{\ell-j};
\]
here $\operatorname*{Op}(a_{m\ell})=\operatorname*{Op}_{{BJ,1}}(a_{m\ell})$ is
the Born-Jordan ordering (\cite{dogo15} and \cite[page 34]{gobook16}).

In the case of a general $n$ we have, by Leibniz's formula,
\begin{equation}
f^{(j)}(0)=\sum_{k=0}^{j}\binom{j}{k}\left(  \frac{i}{2}\right)  ^{j-k}\left(
\frac{1}{2}\right)  ^{k}\left(  \frac{d^{k}}{dt^{k}}\operatorname{sinc}
^{n}\right)  (0). \label{fjo}%
\end{equation}
The derivatives of $\operatorname{sinc}^{n}$ at $t=0$ can be calculated using
Fa\`{a} di Bruno's formula \cite{faa} for the derivatives of the composition
of two functions
\begin{equation}
(g\circ h)^{(k)}(t)=\sum_{\kappa\cdot\alpha=k}\binom{k}{\alpha}g^{(|\alpha
|)}(h(t))\Pi_{\alpha}(t) \label{faa1}%
\end{equation}
where $\kappa=(1,2,...,k)$, $\alpha=(\alpha_{1},\alpha_{2},...,\alpha_{k}
)\in\mathbb{N}^{k}$ and
\[
\Pi_{\alpha}(t)=\left(  \frac{1}{1!}h^{(1)}(t)\right)  ^{\alpha_{1}}\left(
\frac{1}{2!}h^{(2)}(t)\right)  ^{\alpha_{2}}\cdot\cdot\cdot\left(  \frac
{1}{k!}h^{(k)}(t)\right)  ^{\alpha_{k}}.
\]
Choosing $g(t)=x^{n}$ and $h(t)=\operatorname{sinc}(t/2)$ this formula yields
\[
\frac{d^{k}}{dt^{k}}\operatorname{sinc}^{n}(0)=\sum_{\substack{\kappa
\cdot\alpha=k\\|\alpha|\leq n}}\binom{k}{\alpha}\binom{n}{|\alpha|}
|\alpha|!\Pi_{\alpha}(0);
\]
since $\operatorname{sinc}^{(2m+1)}(0)=0$ and $\operatorname{sinc}
^{(2m)}(0)=(-1)^{m}/(2m+1)$ we have
\[
\Pi_{\alpha}(0)=\frac{1}{1!(\alpha_{1}+1)^{\alpha_{1}}2!(\alpha_{2}
+1)^{\alpha_{2}}\cdot\cdot\cdot k!(\alpha_{k}+1)^{\alpha_{k}}}.
\]

\section{Time-frequency Analysis of the nth Born-Jordan kernel}

The Born-Jordan kernel $\Theta(\zeta)$ in \eqref{sincxp} belongs to the space
$W(\mathcal{F} L^{1}, L^{\infty})({\mathbb{R}^{2d}})$, as proved in
\cite{ACHA2018}:

\begin{proposition}
\label{pro2}  The function $\Theta^{1}$ in \eqref{sincxp} belongs to
$W(\mathcal{F} L^{1},L^{\infty})({\mathbb{R}^{2d}})$.
\end{proposition}

The previous property is true for any $\Theta^{n}$, $n\in\mathbb{N}_{+}$, as
shown below.

\begin{proposition}
\label{pron}  For $n\in\mathbb{N}_{+}$, the function $\Theta^{n}$ defined in
\eqref{nCohenkerneln} belongs to the Wiener algebra $W(\mathcal{F}
L^{1},L^{\infty})({\mathbb{R}^{2d}})$.
\end{proposition}

\begin{proof}
The result is attained by induction on $n$. We know that $\Theta^{1}\in
W(\mathcal{F} L^{1},L^{\infty})({\mathbb{R}^{2d}})$ by Proposition \ref{pro2}.
If we assume $\Theta^{n}\in W(\mathcal{F} L^{1},L^{\infty})({\mathbb{R}^{2d}%
})$, for a certain integer $n>1$, we obtain
\[
\Theta^{n+1}=\Theta^{n}\cdot\Theta^{1}\in W(\mathcal{F} L^{1},L^{\infty
})({\mathbb{R}^{2d}})\cdot W(\mathcal{F} L^{1},L^{\infty})({\mathbb{R}^{2d}%
})\hookrightarrow W(\mathcal{F} L^{1},L^{\infty})({\mathbb{R}^{2d}}),
\]
since the Banach space $W(\mathcal{F} L^{1},L^{\infty})({\mathbb{R}^{2d}})$ is
an algebra by pointwise product. This gives the claim.
\end{proof}

In \cite{ACHA2018} it was shown the following property for the chirp function:

\begin{proposition}
\label{pro1} The function $F(\zeta_{1},\zeta_{2})= e^{ 2\pi i \zeta_{1}
\zeta_{2}}$ belongs to $W(\mathcal{F} L^{1},L^{\infty})({\mathbb{R}^{2d}})$.
\end{proposition}

Since $W(\mathcal{F} L^{1},L^{\infty})({\mathbb{R}^{2d}})$ can be
characterized as the space of pointwise multipliers on the Feichtinger algebra
$W(\mathcal{F} L^{1},L^{1})({\mathbb{R}^{2d}})$ \cite[Corollary 3.2.10]%
{feizim}, the result in Proposition \ref{pro1} could also be deduced from
general results about the action of second order characters on the Feichtinger
algebra, cf.\ \cite{fei0,reiter}.

By Proposition \ref{pro1} and by the dilation properties for Wiener amalgam
spaces \eqref{dilAW0} we can state:

\begin{corollary}
\label{cor1} For $\zeta=(\zeta_{1},\zeta_{2})$, consider the function
$F_{J}(\zeta)=F( J \zeta)= e^{- 2\pi i \zeta_{1} \zeta_{2}}$.  Then $F_{J}\in
W(\mathcal{F} L^{1},L^{\infty})({\mathbb{R}^{2d}})$.
\end{corollary}


\section{Smoothness of the Born-Jordan distribution of order $n$}

In the present section we compare the smoothness of the Born-Jordan
distribution of order $n$ with the Wigner distribution. In particular we prove
Theorem \ref{mainteo}.

We begin with the following global result, which in particular implies Theorem
\ref{teo2-zero}.

\begin{theorem}
\label{teo2} Let $f\in\mathcal{S}^{\prime}(\mathbb{R}^{d})$ be a signal, with
$Wf\in M^{p,q}({\mathbb{R}^{2d}})$ for some $1\leq p,q\leq\infty$. Then
\[
Q^{n}f\in M^{p,q}({\mathbb{R}^{2d}})
\]
and moreover
\begin{equation}
\label{eq1}(\nabla_{x}\cdot\nabla_{\omega})^{n} Q^{n}f\in M^{p,q}%
({\mathbb{R}^{2d}}).
\end{equation}

\end{theorem}

Here we used the notation
\[
\nabla_{x}\cdot\nabla_{\omega}:=\sum_{j=1}^{d} \frac{\partial^{2}}%
{\partial{x_{j}}\partial{\omega_{j}}}.
\]

\begin{proof}
We first show $Q^{n} f\in M^{p,q}({\mathbb{R}^{2d}})$. Taking the symplectic
Fourier transform in \eqref{bj} we are reduced to prove that
\[
\Theta^{n} \mathcal{F}_{\sigma}(Wf)=\Theta^{n} Af\in W(\mathcal{F} L^{p}%
,L^{q})
\]
where $\mathcal{F}_{\sigma}(Wf)=Af$ is the ambiguity function of $f$ in
\eqref{ambiguity}. The claim is attained using the product property
\eqref{product}: by Proposition \ref{pron}, the function $\Theta^{n}$ is in
$W(\mathcal{F} L^{1},L^{\infty})$ and by assumption $Wf\in M^{p,q}%
({\mathbb{R}^{2d}})$ so that $\mathcal{F}(Wf)\in W(\mathcal{F} L^{p},L^{q})$,
and therefore $\mathcal{F}_{\sigma}(Wf)(\zeta)=\mathcal{F}(Wf)(J\zeta) \in
W(\mathcal{F} L^{p},L^{q})$ by Proposition \ref{c1}.

We now prove \eqref{eq1}. Taking the symplectic Fourier transform we see that
it is sufficient to prove that
\[
(\zeta_{1}\zeta_{2})^{n}\, \mathrm{sinc}^{n}(\zeta_{1}\zeta_{2})
\mathcal{F}_{\sigma}Wf=\frac{1}{\pi^{n}}\sin^{n}(\pi\zeta_{1}\zeta_{2})
\mathcal{F}_{\sigma}Wf\in W(\mathcal{F} L^{p},L^{q}).
\]
We have,
\begin{equation}
\label{sin}\sin(\pi\zeta_{1}\zeta_{2})=\frac{e^{\pi i \zeta_{1}\zeta_{2}%
}-e^{-\pi i \zeta_{1}\zeta_{2}}}{2i}\in W(\mathcal{F} L^{1},L^{\infty}),
\end{equation}
by Proposition \ref{pro1}, Corollary \ref{cor1} and Proposition \ref{c1}, with
the scaling $\lambda=1/\sqrt{2}$.

Hence, for $n=1$,
\[
\frac{1}{\pi}\sin(\pi\zeta_{1}\zeta_{2}) \mathcal{F}_{\sigma}Wf\in
W(\mathcal{F} L^{p},L^{q})
\]
by the product property \eqref{product}. Assume now that, for a certain
$n\in\mathbb{N}_{+}$,
\[
\frac{1}{\pi^{n}}\sin^{n}(\pi\zeta_{1}\zeta_{2}) \mathcal{F}_{\sigma}Wf\in
W(\mathcal{F} L^{p},L^{q}).
\]
Then
\[
\frac{1}{\pi^{n+1}}\sin^{n+1}(\pi\zeta_{1}\zeta_{2}) \mathcal{F}_{\sigma
}Wf=\underset{\in W(\mathcal{F} L^{1},L^{\infty})}{\underbrace{\frac1\pi
\sin(\pi\zeta_{1}\zeta_{2})}}\cdot\underset{\in W(\mathcal{F} L^{p}%
,L^{q})}{\underbrace{\frac{1}{\pi^{n}}\sin^{n}(\pi\zeta_{1}\zeta_{2})
\mathcal{F}_{\sigma}Wf}} \in W(\mathcal{F} L^{p},L^{q}),
\]
by \eqref{sin} and the product property \eqref{product} again. By induction we
attain the result.
\end{proof}

We are now ready to prove Theorem \ref{mainteo}.

\begin{proof}
[Proof of Theorem \ref{mainteo}]Consider $n\in\mathbb{N}_{+}$. We will apply
Proposition \ref{pro3} to the $2n$-th order operator $P^{n}$, where
$P=\nabla_{x}\cdot\nabla_{\omega}$ in ${\mathbb{R}^{2d}}$. The non
characteristic directions for $P^{n}$ are given by the vectors $\zeta
=(\zeta_{1},\zeta_{2})\in\mathbb{R}^{d}\times\mathbb{R}^{d}$, satisfying
$\zeta_{1}\cdot\zeta_{2}\not =0$. By \eqref{eq1} (with $p=\infty$) we have
\[
WF_{\mathcal{F} L^{q}}(P^{n} Q^{n} f)=\emptyset,
\]
because $\varphi F\in\mathcal{F} L^{q}$ if $\varphi\in C^{\infty}%
_{c}({\mathbb{R}^{2d}})$ and $F\in M^{\infty,q}({\mathbb{R}^{2d}})$ (here
$F=P^{n} Q^{n}f$). Hence we obtain
\[
(z,\zeta)\not \in WF_{\mathcal{F} L^{q}}(P^{n} Q^{n}f),\quad\forall
(z,\zeta)\,\,\mbox{such \,that}\,\, \,\zeta=(\zeta_{1},\zeta_{2}%
),\,\,\zeta_{1}\cdot\zeta_{2}\not =0.
\]
Since $\zeta$ is non characteristic for the operator $P^{n}$, by Proposition
\ref{pro3} we infer
\[
(z,\zeta)\not \in WF_{\mathcal{F} L^{q}_{2n}}(Q^{n}f)
\]
for every $z\in{\mathbb{R}^{2d}}$.
\end{proof}

\begin{proof}
[Proof of Corollary \ref{cor}]Apply Theorem \ref{mainteo} with $q=2$. Indeed,
for $f\in L^{2}(\mathbb{R}^{d})$ Moyal's formula gives $Wf\in L^{2}%
({\mathbb{R}^{2d}})=M^{2,2}(\mathbb{R}^{d})\subset M^{\infty,2}({\mathbb{R}%
^{2d}})$ (cf.\ \eqref{wigner2}). Observe that the $\mathcal{F} L^{2}_{2n}$
wave-front set coincides with the $H^{2n}$ wave-front set.
\end{proof}

The proof of Theorem \ref{teo3} requires Lemma 5.1 in \cite{ACHA2018}:

\begin{lemma}
\label{lemma5.1}  Let $\chi\in C^{\infty}_{c}(\mathbb{R})$. Then the function
$\chi(\zeta_{1} \zeta_{2})$ belongs to $W(\mathcal{F} L^{1},L^{\infty
})({\mathbb{R}^{2d}})$.
\end{lemma}

As announced in the introduction, the smoothing phenomena of the $Q^{n}$
distributions do not involve the whole phase space. We do not have any gain in
the directions $\zeta_{1}\cdot\zeta_{2}=0$ as it comes up clearly from the
proof of the following issue.

\begin{proof}
[Proof of Theorem \ref{teo3}]The pattern is similar to that of Theorem 1.4 in
\cite{ACHA2018}. We detail the main steps for sake of clarity. The idea is to
test the estimate \eqref{test} using rescaled Gaussian functions
$f(x)=\varphi(\lambda x)$, with $\lambda>0$ large parameter. We shall prove
that, restricting to a neighbourhood of $\zeta_{1}\cdot\zeta_{2}=0$, the
constrain $q_{1}\geq q_{2}$ must be satisfied.

An easy computation (see e.g.\ \cite[Formula (4.20)]{grochenig}) yields
\begin{equation}
\label{wignerdil}W(\varphi(\lambda\, \cdot))(x,\omega)=2^{d/2} \lambda^{-d}
\varphi(\sqrt{2}\lambda\, x)\varphi(\sqrt{2}\lambda^{-1}\, \omega).
\end{equation}
For every $1\leq p,q\leq\infty$, the above formula gives
\[
\|W(\varphi(\lambda\, \cdot))\|_{M^{p,q}}=2^{d/2} \lambda^{-d}\| \varphi
(\sqrt{2}\lambda\, \cdot) \|_{M^{p,q}}\|\varphi(\sqrt{2}\lambda^{-1}\, \cdot)
\|_{M^{p,q}}.
\]
By the dilation properties of Gaussians in Lemma \ref{lemma5.2-zero}
\begin{equation}
\label{eqa3}\|W(\varphi(\lambda\, \cdot))\|_{M^{p,q}}\asymp\lambda
^{-2d+d/q+d/p}\quad\mathrm{as}\ \lambda\to+\infty.
\end{equation}
We now study the $M^{p,q}$-norm of the BJDn $Q^{n}(\varphi(\lambda\, \cdot))$.
The idea is to estimate such a norm from below obtaining the same expansion as
in \eqref{eqa3}.
\[
\|Q^{n}(\varphi(\lambda\, \cdot))\|_{M^{p,q}}=\|\mathcal{F}_{\sigma}%
(\Theta^{n}) \ast W(\varphi(\lambda\, \cdot))\|_{M^{p,q}}.
\]
By taking the symplectic Fourier transform and using Lemma \ref{lemma5.1} and
the product property \eqref{product} we have
\begin{align*}
\|\mathcal{F}_{\sigma}(\Theta^{n}) \ast W(\varphi(\lambda\, \cdot
))\|_{M^{p,q}} & \asymp\|\Theta^{n} \mathcal{F}_{\sigma}[ W(\varphi(\lambda\,
\cdot))]\|_{W(\mathcal{F} L^{p},L^{q})}\\
& \gtrsim\|\Theta^{n}(\zeta_{1},\zeta_{2}) \chi(\zeta_{1}\zeta_{2}%
)\mathcal{F}_{\sigma}[ W(\varphi(\lambda\, \cdot))]\|_{W(\mathcal{F}
L^{p},L^{q})}%
\end{align*}
for any $\chi\in C^{\infty}_{c}(\mathbb{R})$ and $n\in\mathbb{N}_{+}$.
Choosing $\chi$ supported in the interval $[-1/4,1/4]$ and $\chi\equiv1$ in
the interval $[-1/8,1/8]$ (the latter condition will be used later), we write
\[
\chi(\zeta_{1} \zeta_{2})=\chi(\zeta_{1} \zeta_{2}) \Theta^{n}(\zeta_{1}%
,\zeta_{2})\Theta^{-n}(\zeta_{1},\zeta_{2})\tilde{\chi}(\zeta_{1} \zeta_{2}),
\]
with $\tilde{\chi}\in C^{\infty}_{c}(\mathbb{R})$ supported in $[-1/2,1/2]$
and $\tilde{\chi}=1$ on $[-1/4,1/4]$, therefore on the support of $\chi$.
Since by Lemma \ref{lemma5.1} the function $\Theta^{-n}(\zeta_{1},\zeta
_{2})\tilde{\chi}(\zeta_{1} \zeta_{2})$ belongs to $W(\mathcal{F}
L^{1},L^{\infty})$, by the product property the last expression can be
estimated from below as
\[
\gtrsim\| \chi(\zeta_{1}\zeta_{2})\mathcal{F}_{\sigma}[ W(\varphi(\lambda\,
\cdot))]\|_{W(\mathcal{F} L^{p},L^{q})}.
\]
We are ended up with the same object already estimated in the proof of Theorem
1.4 in \cite{ACHA2018}, were it was shown that
\begin{equation}
\label{e27}\| \chi(\zeta_{1}\zeta_{2})\mathcal{F}_{\sigma}[ W(\varphi
(\lambda\, \cdot))]\|_{W(\mathcal{F} L^{p},L^{q})}\gtrsim\lambda
^{-2d+d/p+d/q}\quad\mathrm{as}\ \lambda\to+\infty
\end{equation}
Comparing \eqref{e27} with \eqref{eqa3} we obtain the desired conclusion.
\end{proof}

\section{Pseudodifferential Calculus}

The Weyl quantization was introduced by Weyl in \cite{Weyl1927} and is the
$n=0$ case of the Born-Jordan quantization of order $n$ in \eqref{BJNQ}:
\[
{a\in\mathcal{S}^{\prime}({\mathbb{R}^{2d}})}\mapsto\widehat{A}_{\mathrm{W}%
}=\operatorname*{Op}\nolimits_{\mathrm{W}}(a)=\left(  \tfrac{1}{2\pi\hbar
}\right)  ^{d}\int_{{\mathbb{R}^{2d}}} \mathcal{F}_{\sigma}a(z)\widehat{T}%
(z)dz.
\]
Comparing with \eqref{BJNQ}, we infer the symbol relation
\[
\mathcal{F}_{\sigma}a_{BJ,n} \Theta^{n}= \mathcal{F}_{\sigma}a_{W}
\]
(observe that $a_{BJ,n}$ denotes the symbol of $\widehat{A}_{\mathrm{BJ},n}$
whereas $a_{W}$ is the Weyl symbol) that is
\begin{equation}
\label{symbaBJn}a_{BJ,n} \ast\mathcal{F}_{\sigma}(\Theta^{n})=a_{W}.
\end{equation}

Using the weak definition for Weyl operators via the Wigner distribution
\[
\langle Op_{W}(a)f,g\rangle= \langle a, W(g,f)\rangle,\quad a\in
\mathcal{S}^{\prime}({\mathbb{R}^{2d}}),\,\,f,g\in\mathcal{S}(\mathbb{R}^{d})
\]
and the convolution property (whenever is well-defined)
\[
\langle F\ast G,H\rangle=\langle F , H\ast G\rangle
\]
we can also define, for $n\in\mathbb{N}$, the $n$-th Born-Jordan
pseudodifferential operator with symbol $a\in\mathcal{S}^{\prime}%
(\mathbb{R}^{d})$ by
\begin{equation}
\label{BJpseudoN}\langle Op_{BJ,n}(a) f,g\rangle=\langle a, Q^{n}(g,f)\rangle,
\quad f,g\in\mathcal{S}(\mathbb{R}^{d}).
\end{equation}
(Observe that $n=1$ is the standard BJ operator, whereas $n=0$ gives the Weyl one).

We aim at studying continuity properties of such operators and of the related
distributions on modulation spaces.

First, we analyze the quadratic representations $Q^{n}$.

\begin{theorem}
\label{cohenbound}  Assume $s\geq0$, $p_{1},q_{1},p,q \in[1,\infty]$ such
that
\begin{equation}
\label{Ch}2\min\{ \frac1{p_{1}},\frac1{q_{1}}\}\geq\frac1p +\frac1 q.
\end{equation}
If $f\in M^{p_{1},q_{1}}_{v_{s}}(\mathbb{R}^{d})$ the Cohen distribution
$Q^{n}f$, $n\in\mathbb{N}_{+}$, is in $M^{p,q}_{1\otimes v_{s}}({\mathbb{R}%
^{2d}})$, with
\begin{equation}
\label{cohenM}\|Q^{n}f\|_{M^{p,q}_{1\otimes v_{s}}({\mathbb{R}^{2d}})}%
\lesssim\|\Theta^{n}\|_{W(\mathcal{F} L^{1},L^{\infty})({\mathbb{R}^{2d}})}
\|f\|^{2}_{M^{p_{1},q_{1}}_{v_{s}}(\mathbb{R}^{d})}.
\end{equation}

\end{theorem}

\begin{proof}
In \cite[Theorem 1.2]{CNIMRN2018} two of us proved that, if the Cohen kernel
$\theta$, defined in \eqref{Cohenkernel}, is in $M^{1,\infty}({\mathbb{R}%
^{2d}})$, then the related Cohen distribution $Qf$ satisfies
\[
\|Q^{n}f\|_{M^{p,q}_{1\otimes v_{s}}({\mathbb{R}^{2d}})}\lesssim
\|\theta\|_{M^{1,\infty}({\mathbb{R}^{2d}})}\|f\|^{2}_{M^{p_{1},q_{1}}_{v_{s}%
}(\mathbb{R}^{d})}
\]
where the indices $p_{1},q_{1}, p,q \in[1,\infty]$ are related by condition \eqref{Ch}.

By Proposition \ref{pron}, the function $\Theta^{n}$ is in $W(\mathcal{F}
L^{1},L^{\infty})$, so that the BJ kernel $\mathcal{F}_{\sigma}(\Theta^{n})$
is in $M^{1,\infty}({\mathbb{R}^{2d}})$ with $\|\mathcal{F}_{\sigma}%
(\Theta^{n})\|_{M^{1,\infty}}\asymp\|\Theta^{n}\|_{W(\mathcal{F}
L^{1},L^{\infty})}$  and the thesis follows. 
\end{proof}

We write $q^{\prime}$ for the  conjugate exponent of $q\in\lbrack1,\infty]$;
it is defined by $ 1/q+1/q^{\prime}=1$.  The $n$-th Born-Jordan operator
enjoys the same continuity properties as for the $n=1$ case, proved in
\cite[Theorem 1.1]{cgnb}. Indeed, we can state: 

\begin{theorem}
\label{Charpseudo} Consider $1\leq p,q,r_{1},r_{2}\leq\infty$, such that
\begin{equation}
\label{indicitutti}p\leq q^{\prime}%
\end{equation}
\noindent and
\begin{equation}
\label{indiceq}\quad q \leq\min\{r_{1},r_{2},r_{1}^{\prime},r_{2}^{\prime}\}.
\end{equation}
Then the Born-Jordan operator ${Op}_{{BJ,n}}(a)$, from  $\mathcal{S}%
(\mathbb{R}^{d})$ to $\mathcal{S}^{\prime}(\mathbb{R}^{d})$, having  symbol $a
\in M^{p,q}(\mathbb{R}^{2d})$, extends uniquely to a bounded  operator on
$\mathcal{M}^{r_{1},r_{2}}(\mathbb{R}^{d})$, with the estimate
\begin{equation}
\label{stimaA}\|{Op}_{{BJ,n}}(a) f\|_{\mathcal{M}^{r_{1},r_{2}}}
\lesssim\|a\|_{M^{p,q}}\|f\|_{\mathcal{M}^{r_{1},r_{2}}},\quad f\in\mathcal{M}
^{r_{1},r_{2}}.
\end{equation}
Vice versa, if this conclusion holds true, the constraints
\eqref{indicitutti} is satisfied and it must hold
\begin{equation}
\label{stimanew}\max\left\{ \frac1{r_{1}},\frac1{r_{2}},\frac1{r_{1}^{\prime}%
},\frac1{r_{2}^{\prime}}\right\} \leq\frac1q+\frac1p,
\end{equation}
that is \eqref{indiceq} for $p=\infty$. 
\end{theorem}

\begin{proof}
The sufficient conditions are proved by induction. The result holds true for
$n=1$ by Theorem \cite[Theorem 1.1]{cgnb}. Assume now that the result is true
for a certain $n\in\mathbb{N}_{+}$ and observe, by definition \eqref{BJNQ},
that
\[
Op_{BJ,n+1}(a)=Op_{BJ,n}(b),\quad\mbox{with}\quad a=b\ast\mathcal{F}_{\sigma
}\Theta.
\]
The thesis follows by the convolution relation $M^{p,q}({\mathbb{R}^{2d}})\ast
M^{1,\infty}({\mathbb{R}^{2d}})\hookrightarrow M^{p,q}({\mathbb{R}^{2d}})$.

The necessary conditions are obtained arguing exactly as for the case $n=1$,
for details we refer to the proof of Theorem 1.1 in \cite{cgnb}. 
\end{proof}

\section*{Technical notes}

The figures in the introduction were produced using LTFAT (The Large Time-Frequency Analysis Toolbox), cf.~\cite{ltfat} as well as the Time-Frequency Toolbox
(TFTB), distributed under the terms of the GNU Public Licence:
\begin{center}
http://tftb.nongnu.org/
\end{center} 
The bat sonar signal in Figure 3 was recorded as a .mat file in
the latter  toolbox. 
\section*{Acknowledgments}
The authors  would like to thank Professor Jean-Pierre Gazeau, 
for inspiring this work during the wonderful environment of the conference \emph{Quantum Harmonic Analysis and Symplectic Geometry}, April 21-24, 2018, Strobl, AUSTRIA.
Maurice de Gosson has been financed by the grant P27773 of the Austrian research Foundation FWF.  Monika D\"orfler
has been supported by the Vienna Science and Technology Fund (WWTF)
through project MA14-018. 

\vskip0.5truecm

\end{document}